\documentclass{amsart}
\usepackage{amsfonts}

\newtheorem{theorem}{Theorem}
\theoremstyle{plain}

\newtheorem{corollary}{Corollary}

\numberwithin{equation}{section}
\input{tcilatex}

\begin{document}
\title[Superquadracity and Euler Lagrange identity]{Extension of Euler
Lagrange identity by superquadratic power functions}
\author{S. Abramovich}
\address{Depertment of Mathmatich, University of Haifa, Haifa, Israel}
\email{abramos@math.haifa.ac.il}
\author{S. Iveli\'{c}}
\address{Faculty of civil Technpology and Architecture, University of Split,
Croatia.}
\email{sivelic@gradst.hr}
\author{J. Pe\v{c}ari\'{c}}
\address{Faculty of Textile Technology, University of Zagreb, Croatia.}
\email{pecaric@element.hr}
\date{May 26, 2011}

\begin{abstract}
Using convexity and superquadracity we extend in this paper Euler\ Lagrange
identity, Bohr's inequalitiy and the triangle inequality.
\end{abstract}

\maketitle

\section{Generalization of the triangle inequality via convexity}

In \cite{TRST} Theorem 1.1 inequalities related to the Euler Lagrange
identity are proved on Banach space. Using the convexity of $x^{p}$ $p\geq
1, $ $x\geq 0$ we prove in this section a generalization of this theorem for
complex numbers, for which Bohr's inequality is a special case. This gives
us the tools to achieve the main result of Section 2. There we extend the
result\ to the superquadratic functions $x^{p}$ $p\geq 2,$ $x\geq 0$ and
obtain the Euler Lagrange identity as a special case.

\begin{theorem}
\label{Th1} Let $x,$ $y$,$\ a,$ $b$ be complex numbers and let $\mu ,$ $\nu
, $ $\lambda $ $\in $ $%
\mathbb{R}
\backslash 0$ then 
\begin{equation*}
\frac{\left\vert x\right\vert ^{p}}{\mu }+\frac{\left\vert y\right\vert ^{p}%
}{\nu }\geq \frac{\left\vert ax+by\right\vert ^{p}}{\lambda }
\end{equation*}%
holds if
\end{theorem}

(i) \ $\mu >0,$ $\nu >0,$ $\lambda >0$ and 
\begin{equation*}
\left\vert \lambda \right\vert ^{1/\left( p-1\right) }\geq \left\vert \mu
\right\vert ^{1/\left( p-1\right) }\left\vert a\right\vert ^{q}+\left\vert
\nu \right\vert ^{1/\left( p-1\right) }\left\vert b\right\vert ^{q},
\end{equation*}

(ii) $\ \mu <0,$ $\nu >0,$ $\lambda <0$ and%
\begin{equation*}
\left\vert \lambda \right\vert ^{1/\left( p-1\right) }\leq \left\vert \mu
\right\vert ^{1/\left( p-1\right) }\left\vert a\right\vert ^{q}-\left\vert
\nu \right\vert ^{1/\left( p-1\right) }\left\vert b\right\vert ^{q},
\end{equation*}

(iii) \ $\mu >0,$ $\nu <0,$ $\lambda <0$ and%
\begin{equation*}
\left\vert \lambda \right\vert ^{1/\left( p-1\right) }\leq -\left\vert \mu
\right\vert ^{1/\left( p-1\right) }\left\vert a\right\vert ^{q}+\left\vert
\nu \right\vert ^{1/\left( p-1\right) }\left\vert b\right\vert ^{q},
\end{equation*}%
where $p>1$ and $\frac{1}{p}+\frac{1}{q}=1$.

\textbf{Comment:} \bigskip\ Bohr's inequality 
\begin{equation*}
sx^{p}+ty^{p}\geq \frac{1}{\left( s-1\right) s^{p-2}}\left( \left(
s-1\right) x+y\right) ^{p}\geq \frac{1}{2^{p-2}}\left( \left( s-1\right)
x+y\right) ^{p},
\end{equation*}%
when $1<s\leq 2,$ $\frac{1}{s}+\frac{1}{t}=1,$ $p>1$ is a special case of
Theorem \ref{Th1} for $a=s-1$, $b=1,$ $\mu =\frac{1}{s},$ $\nu =\frac{1}{t},$
$\lambda =\left( s-1\right) s^{p-2}$ (see also \cite{ABP}).

\bigskip We first prove a theorem similar Theorem 1.1 in \cite{TRST} but by
dealing with a general integer $n$ instead of $n=2.$ Our proof is completely
different than the proof in \cite{TRST}. It relies on the convexity of $%
f\left( x\right) =x^{p},$ $p>1,$ $x\geq 0.$

\begin{theorem}
\label{Th2} Let $x_{i},$ a$_{i},$ $i=1,...,n$ be complex numbers and $p>1,$ $%
\frac{1}{q}+\frac{1}{p}=1.$ Case (i): If $\mu _{i}>0,$ $i=1,...,n,$ $\lambda
>0,$ then%
\begin{equation}
\sum_{i=1}^{n}\frac{\left\vert x_{i}\right\vert ^{p}}{\mu _{i}}\geq \frac{%
\left\vert \sum a_{i}x_{i}\right\vert ^{p}}{\lambda }  \label{eq1.1}
\end{equation}%
where%
\begin{equation}
\left\vert \lambda \right\vert ^{\frac{1}{p-1}}\geq \sum_{i=1}^{n}\left\vert
\mu _{i}\right\vert ^{\frac{1}{p-1}}\left\vert a_{i}\right\vert ^{q}.
\label{eq1.2}
\end{equation}%
Case (ii): If $\mu _{1}>0$, $\mu _{i}<0,$ $i=2,...,n,$ $\lambda >0,$ then%
\begin{equation}
\sum_{i=1}^{n}\frac{\left\vert x_{i}\right\vert ^{p}}{\mu _{i}}\leq \frac{%
\left\vert \sum a_{i}x_{i}\right\vert ^{p}}{\lambda }  \label{eq1.3}
\end{equation}%
where%
\begin{equation}
\left\vert \lambda \right\vert ^{\frac{1}{p-1}}\leq \left\vert \mu
_{1}\right\vert ^{\frac{1}{p-1}}\left\vert a_{1}\right\vert
^{q}-\sum_{i=2}^{n}\left\vert \mu _{i}\right\vert ^{\frac{1}{p-1}}\left\vert
a_{i}\right\vert ^{q}.  \label{eq1.4}
\end{equation}%
Case (iii): If $\mu _{1}<0$, $\mu _{i}>0,$ $i=2,...,n,$ $\lambda <0,$ then%
\begin{equation*}
\sum_{i=1}^{n}\frac{\left\vert x_{i}\right\vert ^{p}}{\mu _{i}}\geq \frac{%
\left\vert \sum a_{i}x_{i}\right\vert ^{p}}{\lambda }
\end{equation*}%
where $\lambda $\ satisfies (\ref{eq1.4})
\end{theorem}

\begin{proof}
Case (i): It is obvious that it is enough to prove this case\ of the theorem
for $a_{i},\ x_{i}\geq 0,$ $i=1,...,n$ and show that here 
\begin{equation}
\sum_{i=1}^{n}\frac{x_{i}^{p}}{\mu _{i}}\geq \frac{\sum a_{i}x_{i}^{p}}{%
\lambda }  \label{eq1.5}
\end{equation}%
holds if%
\begin{equation}
\lambda ^{\frac{1}{p-1}}\geq \sum_{i=1}^{n}\mu _{i}^{\frac{1}{p-1}}a_{i}^{q}.
\label{eq1.6}
\end{equation}

Let us consider first a more general inequality than (\ref{eq1.5}) where
instead of the function $f(x)=x^{p},$ $p>1,$ $x\geq 0,$ we deal with a
positive strictly increasing convex function $f$ on $\left( 0,\infty \right) 
$\ which satisfies $f^{-1}\left( AB\right) \geq $\ $f^{-1}\left( A\right)
f^{-1}\left( B\right) ,$ $A,$ $B>0.$ In this case we write 
\begin{equation}
\sum_{i=1}^{n}\frac{f(x_{i})}{\mu _{i}}=\sum_{i=1}^{n}Q_{i}f\left(
f^{-1}\left( \frac{f\left( x_{i}\right) }{\mu _{i}Q_{i}}\right) \right) ,
\label{eq1.7}
\end{equation}%
and then by the convexity of $f$\ we get%
\begin{eqnarray}
&&\sum_{i=1}^{n}Q_{i}f\left( f^{-1}\left( \frac{f\left( x_{i}\right) }{\mu
_{i}Q_{i}}\right) \right)  \label{eq1.8} \\
&\geq &\left( \sum_{j=1}^{n}Q_{j}\right) f\left( \frac{%
\sum_{i=1}^{n}Q_{i}f^{-1}\left( \frac{f\left( x_{i}\right) }{\mu _{i}Q_{i}}%
\right) }{\sum_{j=1}^{n}Q_{j}}\right) .  \notag
\end{eqnarray}%
As $f^{-1}\left( AB\right) \geq $\ $f^{-1}\left( A\right) f^{-1}\left(
B\right) $ and $f$\ is increasing we get that 
\begin{eqnarray}
&&\left( \sum_{j=1}^{n}Q_{j}\right) f\left( \frac{\sum_{i=1}^{n}Q_{i}f^{-1}%
\left( \frac{f\left( x_{i}\right) }{\mu _{i}Q_{i}}\right) }{%
\sum_{j=1}^{n}Q_{j}}\right)  \label{eq1.9} \\
&\geq &\left( \sum_{j=1}^{n}Q_{j}\right) f\left( \frac{%
\sum_{i=1}^{n}x_{i}Q_{i}f^{-1}\left( \frac{1}{\mu _{i}Q_{i}}\right) }{%
\sum_{j=1}^{n}Q_{j}}\right) .  \notag
\end{eqnarray}%
Therefore, from (\ref{eq1.7}), (\ref{eq1.8}) and (\ref{eq1.9}) it is enough
to solve the equality%
\begin{equation*}
\left( \sum_{j=1}^{n}Q_{j}\right) f\left( \frac{%
\sum_{i=1}^{n}x_{i}Q_{i}f^{-1}\left( \frac{1}{\mu _{i}Q_{i}}\right) }{%
\sum_{j=1}^{n}Q_{j}}\right) =\frac{f\left( \sum_{i=1}^{n}a_{i}x_{i}\right) }{%
\overline{\lambda }},
\end{equation*}%
in other words to solve 
\begin{equation}
\frac{Q_{i}f^{-1}\left( \frac{1}{\mu _{i}Q_{i}}\right) }{\sum_{j=1}^{n}Q_{j}}%
=a_{i},\qquad i=1,...,n  \label{eq1.10}
\end{equation}%
and then insert 
\begin{equation}
\overline{\lambda }=\left( \sum_{j=1}^{n}Q_{j}\right) ^{-1}  \label{eq1.11}
\end{equation}%
in order for $\overline{\lambda }$\ to satisfy for given $\mu _{i}>0$ and $%
a_{i}\geq 0,$ $i=1,...,n$ the inequality \ \ \ 
\begin{equation}
\sum_{i=1}^{n}\frac{f\left( x_{i}\right) }{\mu _{i}}\geq \frac{f\left(
\sum_{i=1}^{n}a_{i}x_{i}\right) }{\overline{\lambda }}.  \label{eq1.12}
\end{equation}%
Replacing $\overline{\lambda }$\ by 
\begin{equation}
\lambda >\overline{\lambda }=\left( \sum_{j=1}^{n}Q_{j}\right) ^{-1}
\label{eq1.13}
\end{equation}%
inequality (\ref{eq1.12}) holds too.

Now we return to deal with our function $f\left( x\right) =x^{p},$ $p>1,$ $%
x\geq 0.$ This is a nonnegative increasing convex function for $x\geq 0$ and
it satisfies $f^{-1}\left( AB\right) =$\ $f^{-1}\left( A\right) f^{-1}\left(
B\right) $\ for $A,B>0$.

Returning to the proof of (\ref{eq1.5}) under the condition (\ref{eq1.6}) we
obtain from (\ref{eq1.10}) that%
\begin{equation}
Q_{i}\left( \mu _{i}Q_{i}\right) ^{-\frac{1}{p}}\left(
\sum_{j=1}^{n}Q_{j}\right) ^{-1}=a_{i},\qquad i=1,...,n.  \label{eq1.14}
\end{equation}%
Solving (\ref{eq1.14}) we get that

\begin{equation}
Q_{i}=\frac{\mu _{i}^{\frac{1}{p-1}}a_{i}^{q}}{\left( \sum_{j=1}^{n}\mu
_{j}^{\frac{1}{p-1}}a_{j}^{q}\right) ^{p}},\quad i=1,...,n,  \label{eq1.15}
\end{equation}%
and from (\ref{eq1.11}) that%
\begin{equation}
\overline{\lambda }=\left( \sum_{i=1}^{n}Q_{i}\right) ^{-1}=\left(
\sum_{i=1}^{n}\mu _{i}^{\frac{1}{p-1}}a_{i}^{q}\right) ^{p-1}.
\label{eq1.16}
\end{equation}%
Hence from (\ref{eq1.13}), (\ref{eq1.5}) and (\ref{eq1.6}) are proved when $%
a_{i},$\ $x_{i}\geq 0,$\ $i=1,...,n$\ \ and therefore (\ref{eq1.1}) and (\ref%
{eq1.2}) are proved for the complex numbers $x_{i},$ $a_{i},$\ $i=1,...n.$

Case (ii): If $\mu _{1}>0,$ $\mu _{i}<0$, $i=2,...,n$ and $\lambda >0$ we
rewrite (\ref{eq1.3}) as 
\begin{equation}
\frac{\left\vert \sum_{i=2}^{n}a_{i}x_{i}\right\vert ^{p}}{\left\vert
\lambda \right\vert }+\sum_{i=1}^{n}\frac{\left\vert x_{i}\right\vert ^{p}}{%
\left\vert \mu _{i}\right\vert }\geq \frac{\left\vert x_{1}\right\vert ^{p}}{%
\left\vert \mu _{1}\right\vert }.  \label{eq1.17}
\end{equation}%
Let us make the substitutions 
\begin{eqnarray*}
\left\vert \mu _{i}\right\vert &=&\nu _{i},\qquad i=2,...,n,\qquad
\left\vert \mu _{i}\right\vert =\Lambda ,\qquad \left\vert \lambda
\right\vert =\nu , \\
z_{1} &=&\sum_{i=1}^{n}a_{i}x_{i},\qquad z_{i}=x_{i},\qquad i=2,...,n,
\end{eqnarray*}%
and 
\begin{equation*}
x_{1}=\frac{1}{a_{1}}z_{1}+\sum_{i=2}^{n}\left( \frac{-a_{i}}{a_{1}}\right)
z_{i}=\sum_{i=1}^{n}C_{i}z_{i}.
\end{equation*}%
Inequality (\ref{eq1.17}) becomes 
\begin{equation*}
\sum_{i=1}^{n}\frac{\left\vert z_{i}\right\vert ^{p}}{\nu _{i}}\geq \frac{%
\left\vert \sum_{i=1}^{n}C_{i}z_{i}\right\vert ^{p}}{\Lambda }.
\end{equation*}%
Therefore from Case (i) we get that 
\begin{equation*}
\Lambda ^{\frac{1}{p-1}}\geq \sum_{i=1}^{n}\nu _{i}^{\frac{1}{p-1}%
}\left\vert C_{i}\right\vert ^{q}.
\end{equation*}%
In other words (\ref{eq1.3}) holds when 
\begin{equation*}
\left\vert \mu _{1}\right\vert ^{\frac{1}{p-1}}\geq \frac{\left\vert \lambda
\right\vert ^{\frac{1}{p-1}}}{\left\vert a_{1}\right\vert ^{q}}%
+\sum_{i=2}^{n}\left\vert \mu _{i}\right\vert ^{\frac{1}{p-1}}\left\vert 
\frac{a_{i}}{a_{1}}\right\vert ^{q},
\end{equation*}%
which is the same as (\ref{eq1.4}).

The proof of Case (iii) follows immediately from Case (ii).

This completes the proof of the theorem.
\end{proof}

\begin{corollary}
For $n=2$ we get Theorem \ref{Th1} which is Theorem 1.1 in \cite{TRST} for
complex numbers $x_{i},$ $a_{i},$\ $i=1,...n.$
\end{corollary}

\section{Extension of Euler Lagrange type identity}

Now we extend the Euler Lagrange type inequalities by introducing the set of
superquadratic functions and its basic properties. Euler Lagrange identity
is a special case of this extension.

A function $f$ :$[0,b)\rightarrow 
\mathbb{R}
$ is superquadratic provided that for all $x\in \lbrack 0,b)$ there exists a
constant $C_{f}(x)\in 
\mathbb{R}
\ $such that\ the inequality%
\begin{equation}
f(y)\geq f(x)+C_{f}(x)(y-x)+f(\left\vert y-x)\right\vert ,  \label{eq2.1}
\end{equation}%
holds for all $y\in \lbrack 0,b),$ (\cite[Definition 2.1]{AJS}). The
function $f:$ $[0,b)\rightarrow 
\mathbb{R}
$ is subquadratic if $-f$ is supequadratic.

According to \cite[Theorem 2.2]{AJS} the inequality \ \ 
\begin{eqnarray}
&&f\left( \int h(s)d\mu (s)\right)  \label{eq2.2} \\
&\leq &\int f(h(s))-f\left( \left\vert h(s)-\int h(s)d\mu (s)\right\vert
\right) d\mu (s)  \notag
\end{eqnarray}%
holds for all probability measures $\mu $ and all nonnegative $\mu $%
-integrable $h$, if and only if $f$ is superquadratic.

The discrete version of (\ref{eq2.2}) is 
\begin{equation}
f\left( \sum_{i=1}^{n}\alpha _{i}x_{i}\right) \leq \sum_{i=1}^{n}\alpha
_{i}\left( f\left( x_{i}\right) -f\left( \left\vert
x_{i}-\sum_{j=1}^{n}\alpha _{j}x_{j}\right\vert \right) \right) ,\quad
\label{eq2.3}
\end{equation}%
$x_{i}\in \lbrack 0,b),\quad \alpha _{i}\geq 0,\quad 1=i,...,n,\quad
\sum_{i=1}^{n}\alpha _{i}=1.$

The power functions $f\left( x\right) =x^{p},$ $x\geq 0,$ are convex and
superquadratic for $p\geq 2,$ and convex and subquadratic for $1\leq $ $%
p\leq 2$. Inequalities$\ $(\ref{eq2.1})$,\ $(\ref{eq2.2}) and (\ref{eq2.3})
reduce to inequalities for the function $f\left( x\right) =x^{2}.$

Now we use (\ref{eq2.3}) in order to get the Euler Lagrange type inequality.

\begin{theorem}
\bigskip \label{Th3} Let $x_{i}\geq 0,$ $a_{i}\geq 0,$ $\mu _{i}>0,$ $%
i=1,...,n,$ $p\geq 2,$ $\frac{1}{p}+\frac{1}{q}=1.$ Then%
\begin{eqnarray}
&&  \label{eq2.4} \\
\sum_{i=1}^{n}\frac{x_{i}^{p}}{\mu _{i}} &\geq &\frac{\left( \sum
a_{i}x_{i}\right) ^{p}}{\left( \sum_{j=1}^{n}\mu _{j}^{\frac{1}{p-1}%
}a_{j}^{q}\right) ^{p-1}}  \notag \\
&&+\frac{\sum_{i=1}^{n}\mu _{i}^{\frac{1}{p-1}}a_{i}^{q}}{\left(
\sum_{j=1}^{n}\mu _{j}^{\frac{1}{p-1}}a_{j}^{q}\right) ^{p}}\left(
\left\vert \left( \frac{1}{a_{i}\mu _{i}}\right) ^{\frac{1}{p-1}}\left(
\sum_{j=1}^{n}\mu _{j}^{\frac{1}{p-1}}a_{j}^{q}\right)
x_{i}-\sum_{j=1}^{n}a_{j}x_{j}\right\vert \right) ^{p}.  \notag
\end{eqnarray}%
If $1<p\leq 2$ \ the inverse of (\ref{eq2.4}) holds.
\end{theorem}

\begin{proof}
In Theorem \ref{Th2} we showed that for $x_{i}\geq 0,$ $a_{i}\geq 0,$\ $\mu
_{i}>0,$ $i=1,...,n.$ inequalities (\ref{eq1.5}) and (\ref{eq1.6}) hold.
There%
\begin{equation}
\sum_{i=1}^{n}Q_{i}\left( A_{i}\right) ^{p}=\sum_{i=1}^{n}\frac{\left\vert
x_{i}\right\vert ^{p}}{\mu _{i}}  \label{eq2.5}
\end{equation}%
where 
\begin{equation}
Q_{i}=\frac{\mu _{i}^{\frac{1}{p-1}}a_{i}^{q}}{\left( \sum_{j=1}^{n}\mu
_{j}^{\frac{1}{p-1}}a_{j}^{q}\right) ^{p}},\qquad i=1,...,n,  \label{eq2.6}
\end{equation}%
\begin{equation}
A_{i}=\frac{1}{\left( a_{i}\mu _{i}\right) ^{\frac{1}{p-1}}}\left(
\sum_{j=1}^{n}\mu _{j}^{\frac{1}{p-1}}a_{j}^{q}\right) x_{i},\qquad
i=1,,...,n  \label{eq2.7}
\end{equation}%
and%
\begin{equation}
\frac{\sum_{i=1}^{n}Q_{i}A_{i}}{\sum_{j=1}^{n}Q_{j}}%
=\sum_{i=1}^{n}a_{i}x_{i}.  \label{eq2.8}
\end{equation}%
Therefore, as $f\left( x\right) =x^{p},$ $p\geq 2,$ $x\geq 0$ is
superquadratic, (\ref{eq2.3}) becomes by inserting (\ref{eq2.6})-(\ref{eq2.8}%
) 
\begin{eqnarray}
&&\sum_{i=1}^{n}Q_{i}\left( A_{i}\right) ^{p}  \label{eq2.9} \\
&=&\frac{\sum_{i=1}^{n}\mu _{i}^{\frac{1}{p-1}}a_{i}^{q}\left( \left( \frac{1%
}{a_{i}\mu _{i}}\right) ^{\frac{1}{p-1}}\left( \sum_{j=1}^{n}\mu _{j}^{\frac{%
1}{p-1}}a_{j}^{q}\right) x_{i}\right) ^{p}}{\left( \sum_{j=1}^{n}\mu _{j}^{%
\frac{1}{p-1}}a_{j}^{q}\right) ^{p}}  \notag \\
&\geq &\frac{\left( \sum_{i=1}^{n}a_{i}x_{i}\right) ^{p}}{\left(
\sum_{j=1}^{n}\mu _{j}^{\frac{1}{p-1}}a_{j}^{q}\right) ^{p-1}}  \notag \\
&&+\frac{\sum_{i=1}^{n}\mu _{i}^{\frac{1}{p-1}}a_{i}^{q}}{\left(
\sum_{j=1}^{n}\mu _{j}^{\frac{1}{p-1}}a_{j}^{q}\right) ^{p}}\left(
\left\vert \left( \frac{1}{a_{i}\mu _{i}}\right) ^{\frac{1}{p-1}}\left(
\sum_{j=1}^{n}\mu _{j}^{\frac{1}{p-1}}a_{j}^{q}\right)
x_{i}-\sum_{i=1}^{n}a_{j}x_{j}\right\vert \right) ^{p}.  \notag
\end{eqnarray}

Hence from (\ref{eq2.5}) and (\ref{eq2.9}) we get\ that (\ref{eq2.4}) holds.

If $1<p\leq 2$ then $f\left( x\right) =x^{p},$ $x\geq 0$ is a subquadratic
function, therefore the reverse of (\ref{eq2.4}) holds.
\end{proof}

\begin{corollary}
\label{Cor2}In case n=2 we get that%
\begin{eqnarray}
\frac{x^{p}}{\mu }+\frac{y^{p}}{\nu } &\geq &\frac{\left( ax+by\right) ^{p}}{%
\left( \mu ^{\frac{1}{p-1}}a^{q}+\nu ^{\frac{1}{p-1}}b^{q}\right) ^{p-1}}
\label{eq2.10} \\
&&+\mu ^{\frac{1}{p-1}}a^{q}\left( \left\vert \left( \frac{1}{a\mu }\right)
^{\frac{1}{p-1}}x-\frac{ax+by}{\mu ^{\frac{1}{p-1}}a^{q}+\nu ^{\frac{1}{p-1}%
}b^{q}}\right\vert \right) ^{p}  \notag \\
&&+\nu ^{\frac{1}{p-1}}b^{q}\left( \left\vert \left( \frac{1}{\nu b}\right)
^{\frac{1}{p-1}}y-\frac{ax+by}{\mu ^{\frac{1}{p-1}}a^{q}+\nu ^{\frac{1}{p-1}%
}b^{q}}\right\vert \right) ^{p}  \notag
\end{eqnarray}%
In particular if $f(x)=x^{2},$ $n=2$ as Inequality (\ref{eq2.4}) reduces to
equality we get from (\ref{eq2.10}) that 
\begin{equation*}
\frac{x^{2}}{\mu }+\frac{y^{2}}{\nu }=\frac{\left( ax+by\right) ^{2}}{\mu
a^{2}+\nu b^{2}}+\frac{\left( \nu bx-a\mu y\right) ^{2}}{\mu \nu \left( \mu
a^{2}+\nu b^{2}\right) },
\end{equation*}%
which is Euler Lagrange type identity.
\end{corollary}

\begin{corollary}
\label{Cor3} From Theorem \ref{Th3} as $f\left( x\right) =x^{p},$ $1<p\leq 2$
is both subquadratic and convex, we get that 
\begin{eqnarray*}
0 &\leq &\sum_{i=1}^{n}\frac{x_{i}^{p}}{\mu _{i}}-\frac{\left( \sum
a_{i}x_{i}\right) ^{p}}{\left( \sum_{i=1}^{n}\mu _{i}^{\frac{1}{p-1}%
}a_{i}^{q}\right) ^{p-1}} \\
&&\leq \frac{\sum_{i=1}^{n}\mu _{i}^{\frac{1}{p-1}}a_{i}^{q}}{\left(
\sum_{j=1}^{n}\mu _{j}^{\frac{1}{p-1}}a_{j}^{q}\right) ^{p}}\left(
\left\vert \left( \frac{1}{a_{i}\mu _{i}}\right) ^{\frac{1}{p-1}}\left(
\sum_{j=1}^{n}\mu _{j}^{\frac{1}{p-1}}a_{j}^{q}\right)
x_{i}-\sum_{j=1}^{n}a_{j}x_{j}\right\vert \right) ^{p}.
\end{eqnarray*}
\end{corollary}

\bigskip

\bigskip

\end{document}